\documentclass[12pt]{amsart}

\usepackage{amsfonts,amssymb,amsmath,textcomp}

 2

\usepackage{amsmath}
\usepackage{amssymb}
\usepackage{amsthm}
\usepackage[mathscr]{euscript}
\usepackage{textcomp}
\usepackage{amsmath}
\usepackage{amssymb}
\usepackage{amsthm}

\newtheorem{thm}{Theorem}[section]

\newtheorem{lem}{Lemma}
\newtheorem{rem}{Remark}

\newcommand{\Z}{\mathbb{Z}}

\begin{document}
\title
{$3k-4$ theorem for ordered groups}
\author{Prem Prakash Pandey	}

\address[Prem Prakash Pandey]{School of Mathematical Sciences,
                                           NISER Bhubaneswar (HBNI),
                                            Jatani, Khurda-650 052, India.}
                                            
\email{premshivaganga@gmail.com}


\date{\today}

\maketitle

\begin{abstract}
Recently, G. A. Freiman, M. Herzog, P. Longobardi, M. Maj proved two `structure theorems' for ordered groups \cite{FHLM}. We give elementary proof of these two theorems.
\end{abstract}

\section{Introduction}
For any group $G$ (written multiplicatively) and a subset $S$ of $G$ we define $S^2=\{ab:a,b \in S \}$. Then, the main theorem of \cite{FHLM} is
\begin{thm}\label{Prem}[Theorem 1.3, \cite{FHLM}]
Let $G$ be an ordered group and $S$ be a finite subset of $G$. If $|S^2| \leq 3|S|-3$ then the subgroup generated by $S$ is an abelian subgroup of $G$.
\end{thm}
As a corollary to Theorem \ref{Prem}, they deduce a $3k-4$ type theorem for ordered groups.
\begin{thm}\label{Prem1}[Corollary 1.4, \cite{FHLM}]
Let $G$ be an ordered group and $S$ be a finite subset of $G$ with $|S|=k \geq 3$. If $|S^2| \leq 3|S|-4$, then there exist two commuting elements $x,y$ such that $S \subset \{yx^i: 0 \leq i \leq N \} $ for $N=|S^2|-|S|$.
\end{thm}
We give elementary proofs of Theorem \ref{Prem} and Theorem \ref{Prem1}. We shall always assume that $G$ is an ordered group and $S$ is a finite subset of $G$ with $k$ elements. We shall write $S=\{x_1, \ldots, x_k\}$ and assume that $x_1< \ldots <x_k$.

\section{Proofs}
As in the case of integers, the following inequality holds:
\begin{equation}\label{e1}
|S^2| \geq 2|S|-1.
\end{equation}
In equation (\ref{e1}) equality holds only if $S$ is a geometric progression $\{yx^i: 0 \leq i \leq k \} $, with $x,y$ two commuting elements of $G$.

\begin{lem}\label{L2}
If $S$ is not a geometric progression then $|S^2| \geq 2|S|.$
\end{lem}
\begin{proof}
Let $S=\{x_1< \ldots <x_k \}$. Certainly 
$$x_1x_1<x_1x_2< \ldots <x_1x_k<x_2x_k< \ldots <x_kx_k$$ are $2|S|-1$ distinct elements in $S^2$.
If $|S^2| < 2|S|$ then
$$\{x_1x_1,x_1x_2, \ldots ,x_1x_k,x_2x_k, \ldots ,x_kx_k\} = S^2.$$
Now, consider the elements $x_2x_1<x_2x_2< \ldots <x_2x_k$. All these elements are in $S^2$ and $x_1x_1<x_2x_1, \ldots, x_2x_{k-1}<x_2x_k$. Thus we must have 
$$x_2x_1=x_1x_2, x_2x_2=x_1x_3, x_2x_3=x_1x_4, \ldots, x_2x_{k-1}=x_1x_k.$$
From the above relations it follows that $x_1$ and $x_2$ commute and for $i>2$, $x_i$ is contained in the subgroup generated by $x_1, \ldots , x_{i-1}$. Consequently we get that each $x_i$ commutes with each $x_j$ for $i,j=1, \ldots, k$.

Put $y=x_1,x=x_2x_1^{-1}$, then $x$ and $y$ commute and $S=\{y,xy,x^2y, \ldots, x^{k-1}y\}$.\\
Thus, if $S$ is not a geometric progression then $|S^2| \geq 2|S|$.
\end{proof}

\begin{proof}[Proof of Theorem \ref{Prem1}]
We shall use induction on $k$. For $k=3$, we have $|S^2| \leq 5$. We have five distinct elements $x_1^2<x_1x_2<x_2^2<x_2x_3<x_3^2$ in $S^2$. Since $x_1x_3 \in S^2$, so $x_1x_3$ must equal to one of these five elements. Using the order relation, we get $x_1x_3=x_2^2 $. Similarly, we get $x_1x_2=x_2x_1$. Let $y=x_1$ and $x=x_2x_1^{-1}$. Then $x$ and $y$ commute and $S=\{y,yx,yx^2\}.$\\ 
Now we assume that $k \geq 4$ and the theorem is true for any subset $T$ of $G$ with $|T|\leq k-1$. Put $T=\{x_1, \ldots, x_{k-1} \}.$\\
Case (1): $|T^2| \leq 3|T|-4$.\\
By induction hypothesis, there are commuting elements $x,y$ such that $T \subset \{yx^j: j=0, 
\ldots, M\}$ with $M=|T^2|-|T|$.\\
In case $x_kT \cap T^2=\emptyset$, then, taking $x_k^2$ in account, we see that $|S^2| \geq |T^2|+(|T|+1).$ Since $|T^2| \geq 2|T|-1$, we immediately obtain $|S^2| \geq 3|S|-3,$ which contradicts the hypothesis. Thus, we get $x_kT \cap T^2 \neq \emptyset$. Consequently, there are $yx^i,yx^u,yx^v \in T$ such that $x_kyx^i=yx^uyx^v $. This gives $x_k=yx^{(u+v-i)}$ and $S \subset \{yx^j: j=0, 
\ldots, M'\}$ with $M'=\max \{M,u+v-i\}$. Clearly the map $yx^j \mapsto j$ gives a $2-isomorphism$ of $S$ with a subset of $\Z$. From the Freiman's $3k-4$-theorem for integers, it follows that $M' \leq N$, and  the theorem is proved.\\
Case (2): $|T^2| \geq 3|T|-3=3|S|-6$. Using the order relation of $G$ we see that the elements $x_k^2$ and $x_kx_{k-1}$ of $S^2$ are not in $T^2$. Consider the element $x_{k-1}x_{k}$ of $S^2$. If $x_{k-1}x_k \neq x_kx_{k-1}$ then we get $|S^2| \geq |T^2|+3$, which contradicts the hypothesis. So, we obtain $x_{k-1}x_k = x_kx_{k-1}$. Next, we consider the element $x_{k-2}x_k$ of $S^2$. If $x_{k-2}x_k\neq x_{k-1}^2$, then we already get $|S^2| \geq |T^2|+3$, leading to a contradiction. Similarly it follows that $x_kx_{k-2}=x_{k-1}^2$. Thus we have $$x_{k-1}x_k = x_kx_{k-1}, x_{k-2}x_k= x_{k}x_{k-2}=x_{k-1}^2.$$
Put $y=x_k, x=x_{k-1}x_k^{-1}$. Then $x$ and $y$ commute and $x_k=y, x_{k-1}=yx,x_{k-2}=yx^2$. Considering the elements $x_{k-3}x_k,x_{k-4}x_k, \ldots, x_1x_k$ successively we see that each of $x_i$ is of the form $yx^{t_i}$. Clearly $S$ is $2-isomorphic$ to the subset $\{ t_i : 1 \leq i \leq k \}$ of $\Z$. Now the theorem follows from the Freiman's $3k-4$-theorem for integers.
\end{proof}

\begin{proof}
[Proof of Theorem \ref{Prem}] We shall use induction on $k$. For $k=1,2$, the theorem holds trivially.
Now, let $k \geq 3$ and assume that the theorem is true for any set $T$ with $|T|\leq k-1$. Put $T=\{x_1, \ldots, x_{k-1} \}.$\\
Case (1): $|T^2| \leq 3|T|-3$.\\
By induction hypothesis, $T$ generates a commutative subgroup. If $x_kT \cap T^2 \neq \emptyset$ or $Tx_k \cap T^2 \neq \emptyset$ then $x_k$ lies in the subgroup generated by $T$. Consequently, $S$ generates a commutative subgroup. So we assume that $x_kT \cap T^2 = \emptyset$ and $Tx_k \cap T^2 = \emptyset$. Using the order relation, we see that $x_k^2 \not \in T^2 \cup x_kT$, so we obtain
\begin{equation}\label{e12}
|S^2| \geq |T^2|+|T|+1.
\end{equation}
If $T$ is not a geometric progression then, using Lemma \ref{L2} in (\ref{e12}), we see that $|S^2| \geq 3|S|-2$, which contradicts the hypothesis. Thus, $T$ must be a geometric progression.\\ 
Next, observe that if $x_kT \neq Tx_k$ then we have an element in $Tx_k$ which is not in $T^2 \cup x_kT \cup \{x_k^2\}$. This leads to 
\begin{equation}\label{e2}
|S^2| \geq |T^2|+|T|+1+1.
\end{equation}
From this one obtains $|S^2| \geq 3|S|-2$, which contradicts the hypothesis. Thus, we must have $x_kT = Tx_k$. Now using the order relation we see that $x_k$ commutes with all the elements of $T$ and consequently $S$ generates an abelian group.\\
Case (2): $|T^2| > 3|T|-3$.\\
As in the proof of Theorem \ref{Prem1} (the arguments used in Case (2)) we see that either $|S^2| \geq |T^2|+3$ or $S=\{yx^{t_i}: 1 \leq i \leq k \}$ with commuting elements $x$ and $y$. The former leads to a contradiction and hence we get $S=\{yx^{t_i}: 1 \leq i \leq k \}$ with commuting elements $x$ and $y$. This proves the theorem.
\end{proof}

\begin{rem}
From the proof of Theorem \ref{Prem1} it is clear that the subgroup generated by $S$ (with $|S|>2$) is, in fact, generated by $|S|-1$ or less elements.
\end{rem}

\end{document}